\documentclass[oneside,11pt,reqno]{amsart}
\usepackage{amssymb,amsmath,amsthm,bbm,enumerate,mdwlist,url,multirow,hyperref,amsthm}
\usepackage[pdftex]{graphicx}
\usepackage[shortlabels]{enumitem}

\theoremstyle{definition}
\newtheorem{definition}{Definition}%Extra square-bracket argument achieves that the numbering is the same as for definition (single uniform counter). 
\theoremstyle{theorem}
\newtheorem{proposition}[definition]{Proposition}
\newtheorem{lemma}[definition]{Lemma}
\newtheorem{theorem}[definition]{Theorem}

\numberwithin{equation}{section}
\theoremstyle{remark}
\newtheorem{remark}[definition]{Remark}

\newtheorem{example}[definition]{Example}
\def\PP{\mathsf P}
\def\EE{\mathsf E}
\def\QQ{\mathsf Q}
\def\BB{\mathcal B}
\def\FF{\mathcal F}

\def\GG{\mathcal G}
\def\Bo{\BB^{(\omega)}_q}
\def\Wl{W^{(\lambda+q)}}
\def\Wq{W^{(q)}}
\def\Wm{W^{(\lambda)}}
\def\Ho{\mathcal{H}^{(\omega)}_q}
\def\Hog{\mathcal{H}^{(\gamma\omega)}_q}

\def\Lo{L^{(\omega)}_q}
\def\Log{L^{(\gamma \omega)}_q}

\def\Id{\cdot}
\def\Ha{\mathcal{H}_\alpha}
\def\La{L_\alpha}
\def\AA{\mathcal A}
\begin{document}
\title{First passage upwards for state dependent-killed spectrally negative L\'evy processes}

\author{Matija Vidmar}
\address{Department of Mathematics, University of Ljubljana, Slovenia}
\email{matija.vidmar@fmf.uni-lj.si}

\begin{abstract}
For a spectrally negative L\'evy process $X$, killed according to a rate that is a function $\omega$ of its position, we complement the recent findings of \cite{zbigniew-bo} by analysing (in greater generality) the exit probability of the one-sided upwards-passage problem. When $\omega$ is strictly positive, this problem is related to the determination of the Laplace transform of the first passage time upwards for $X$ that has been time-changed by the inverse of the additive functional $\int_0^\cdot \omega(X_u)du$. In particular our findings thus shed extra light on related results concerning first passage times downwards (upwards) of continuous state branching processes (spectrally negative positive self-similar  Markov processes).
\end{abstract}

\thanks{Financial support from the Slovenian Research Agency is acknowledged (research core funding No. P1-0222).}



\maketitle

\section{Introduction}

Let $X=(X_t)_{t\in [0,\infty)}$ be a spectrally negative L\'evy process (snLp) under the probabilities $(\PP_x)_{x\in \mathbb{R}}$. This means that $X$ is a  c\`adl\`ag, real-valued process with no positive jumps and non-monotone paths, which, under $\PP_0$,  a.s. vanishes at zero and has stationary independent increments; furthermore, for each $x \in \mathbb{R}$, the law of $X$ under $\PP_x$ is that of $x+X$ under $\PP_0$. We refer to \cite{bertoin,kyprianou,doney,sato} for the general background on (the fulctuation theory of) L\'evy processes and to \cite[Chapter~VII]{bertoin} \cite[Chapter~8]{kyprianou} \cite[Chapter~9]{doney}  \cite[Section~9.46]{sato} for snLp in particular. As usual we set $\PP:=\PP_0$.  For $c\in \mathbb{R}$ denote next by $\tau_c^+:=\inf\{t\in (0,\infty):X_t>c\}$ the first hitting time of the set $(c,\infty)$ by the process $X$. Further, let $q\in [0,\infty)$ and let  $e_q$ be an exponentially with mean $q^{-1}$ distributed random variable ($e_0=\infty$ a.s.) independent of $X$ (under $\PP_x$ for all $x\in \mathbb{R}$). Finally, let $\omega:\mathbb{R}\to [0,\infty)$ be Borel measurable and locally bounded.  Then, for real $x\leq c$, we will be interested in the quantity\footnote{Throughout we will write $\QQ[W]$ for $\EE_\QQ[W]$, $\QQ[W;A]$ for $\EE_\QQ[W\mathbbm{1}_A]$, $\QQ[W\vert A]$ for $\EE_\QQ[W\vert A]$ ($A$ an event) and $\QQ[W\vert \mathcal{G}]$ for $\EE_\QQ[W\vert \mathcal{G}]$ ($\GG$ a sub-$\sigma$-field).} 
\begin{equation}\label{what-we-are-after}
\Bo(x,c):=\PP_x \left[\exp\left\{-\int_0^{\tau_c^+}\omega(X_u)du\right\};\tau_c^+<e_q\right].
\end{equation}
This may be interpreted as the ultimate passage  probability of $X$, killed at $e_q$, over the level $c$, when started at $x$, under ``$\omega$-killing'', i.e. when $X$ is killed (in addition to being killed at the time $e_q$) according to a rate that depends on the position of $X$ and that is  given by the function $\omega$. Of course $\Bo(x,c)=\mathcal{B}^{(\omega+q)}_0(x,c)$, but it will be convenient to keep the independent exponential killing separate.

Assume now that $\omega$ is strictly positive everywhere. Our main motivation for the interest in \eqref{what-we-are-after} comes from its involvement in the solution of the first passage problem upwards for the process that we will denote by $Y=(Y_s)_{s\in [0,\infty)}$, and that is defined as follows. Setting $\zeta:=\int_0^{e_q} \omega(X_u)du$ (see \cite{doring} for conditions on the finiteness/divergence of this integral in the case $q=0$, i.e. $e_q=\infty$), then
\begin{equation}\label{eq:Y}
\text{for $s\geq\zeta$, $Y_s=\partial$, where $\partial$ is some ``cemetery'' state, whilst for $s\in [0,\zeta)$, $Y_s=X_{\rho_s}$,}
\end{equation}
with $$\rho_s:=\inf \left\{t\in [0,\infty):\int_0^t\omega(X_u) du>s\right\}\text{ for }s\in [0,\infty).$$ Notice  that  $\rho$ is continuous (because $\omega$ is strictly positive, and hence $\int_0^\cdot\omega(X_u) du$ strictly increasing) and it is strictly increasing where it is finite  (because $\omega$ is locally bounded, and hence  $\int_0^\cdot \omega(X_u)du$ continuous). Thus the paths of $Y$ up to $\zeta$ are the same as the paths of $X$ up to $e_q$ -- modulo the random time change $\rho$. Also, if $\FF=(\FF_t)_{t\in [0,\infty)}$ is any filtration relative to which $X$ is adapted and has independent increments, with $e_q$ independent of $\FF_\infty$, then thanks to the strong Markov property of $X$ and the memoryless property of the exponential distribution, the process $Y$ is Markovian with state space $(\mathbb{R},\mathcal{B}_\mathbb{R})$ and life-time $\zeta$ under the probabilities $(\PP_y)_{y\in \mathbb{R}}$ and in the filtration $\GG=(\GG_s)_{s\in [0,\infty)}:=(\FF_{\rho_s}\lor \sigma(\{\{\rho_u<e_q\}:u\in [0,s]\}))_{s\in [0,\infty)}$, in the precise sense that it is $\GG$-adapted and that for any Borel measurable $h:\mathbb{R}\to [0,\infty]$, and any $y\in \mathbb{R}$, $t\in [0,\infty)$, a.s.-$\PP_y$, $\PP_y[h(Y_{t+s})\mathbbm{1}_{\{t+s<\zeta\}}\vert \GG_t]=\PP_{Y_t}[h(Y_s);s<\zeta]\mathbbm{1}_{\{t<\zeta\}}$. (Of course if in addition one has a $\PP_\partial$ such that $Y_t=\partial$ for all $t\in [0,\infty)$ a.s.-$\PP_\partial$, then as a consequence $Y$ is also simply Markovian with state space $(\mathbb{R}\cup \{\partial\},\sigma(\mathcal{B}_\mathbb{R}\cup \{\{\partial\}\}))$ and infinite life-time.)  

By way of example, when $\omega=\exp$ and $\partial=-\infty$, then, under the probabilities $(\PP_{\log x})_{x\in (0,\infty)}$, $S:=\exp(Y)$ is a spectrally negative positive self-similar Markov process (pssMp) absorbed at the origin, with index of self-similarity $1$, associated to $X$ via the Lamperti transform for pssMp \cite[Theorem 13.1]{kyprianou}. And, for $c\in (0,\infty)$, on $\{\tau_{\log c}^+<e_q\}$, $\int_0^{\tau_{\log c}^+}\exp(X_u)du$ is the first time that $S$ hits the set $(c,\infty)$, the latter time being $=\infty$ on the complement of $\{\tau_{\log c}^+<e_q\}$. See Example~\ref{example:two}. Similarly, for $c\in (0,\infty)$, with $\omega(x)=\frac{1}{\vert x\vert}$ for $x\in (-\infty,-c]$ and again $\partial=-\infty$, $-Y^{\tau_{-c}^+}$ becomes, under the probabilities $(\PP_{-x})_{x\in [0,\infty)}$, a continuous state branching process (csbp) $B$ stopped on hitting the set $[0,c)$, where $B$ is the csbp associated to $-X$ under the Lamperti transform for csbp \cite[Theorem~12.2]{kyprianou}.\footnote{We are forced to stop at $\tau_{-c}^+$ in order to remain in the setting of a locally bounded $\omega$, which is an assumption that remains in force throughout this paper.} And, on $\{\tau_{-c}^+<e_q\}$, $\int_0^{\tau_{-c}^+}\frac{du}{\vert X_u\vert}$ is the first time $B$ hits $[0,c)$, the latter time being $=\infty$ on the complement of $\{\tau_{-c}^+<e_q\}$. See Example~\ref{example:csbp}.

More generally, denote for $d\in \mathbb{R}$, by $T_d^+:=\inf\{s\in (0,\infty):Y_s\in (d,\infty)\}$ the first hitting time of the set $(d,\infty)$ by the process $Y$.  Then, for $\gamma\in [0,\infty)$, and real $y\leq d$, under $\PP_y$, the Laplace transform of $T_d^+=\int_0^{\tau_{d}^+}\omega(X_u)du$ on $\{T_d^+<\zeta\}=\{\tau_{ d}^+<e_q\}$, at the point $\gamma$, is given simply by 
\begin{equation}\label{eq:pssMp-Laplace}
\PP_y[e^{-\gamma T_d^+};T_d^+<\zeta]=\mathcal{B}^{(\gamma\omega)}_q(y,d).
\end{equation}
Moreover, knowledge of this expression automatically furnishes also the joint Laplace transform of $\tau_{ d}^+$ and $T_d^+$: if further $p\in [0,\infty)$, then $\PP_y[e^{-\gamma T_d^+-p\tau_{d}^+};T_d^+<\zeta]=\mathcal{B}^{(\omega)}_{q+p}(y, d)$. 

Literature-wise, fluctuation results for the ``$\omega$-killed'' snLp $X$ have been the subject of the substantial recent study of \cite{zbigniew-bo} to which the reader is referred for a further review of existing and related results as well as extra  motivation for considering such processes. 

Our contribution is only a small complement to the findings of \cite{zbigniew-bo}, but still one that seems to deserve recording. To be precise, \cite{zbigniew-bo} provides information on the one-sided upwards passage problem when $\omega$ is constant on $(-\infty,0]$ (see \cite[Section~2.4]{zbigniew-bo}); we will extend this  to a far more general class of functions $\omega$. In this class, the solution to \eqref{what-we-are-after} will be given in terms of a function $\Ho$ that will be found to solve  (uniquely)  a natural convolution equation on the real line involving the $q$-scale function of $X$ (Theorem~\ref{theorem}). In contrast to the two-sided exit problem, where the pertinent convolution equation is on the nonnegative half-line \cite[Eq.~(1.2)]{zbigniew-bo}, this introduces some extra finiteness issues, making the analysis slightly more delicate. The function $\Ho$ will also be associated with a family of (local) martingales involving the process $Y$ (Proposition~\ref{proposition:martingale}).

To avoid unnecessary repetition we turn to the results and their proofs presently in  Section~\ref{section:results}, after briefly introducing some necessary further notation and recalling some known facts in Section~\ref{section:notation}.  Section~\ref{illustration} concludes by illustrating the findings in the context of determining the optimal level at which to sell an asset whose price process is given by the exponential of the process $Y$ from \eqref{eq:Y}.

\section{Further notation and some preliminaries}\label{section:notation}
We denote by $\psi$ the Laplace exponent of $X$, $\psi(c):=\log \PP[e^{cX_1}]$ for $c\in [0,\infty)$; and by $\Phi$ its right-continuous inverse, $\Phi(p):=\inf\{c\in [0,\infty):\psi(c)>p\}$ for $p\in [0,\infty)$; $\psi$ is strictly convex and continuous, $\lim_\infty\psi=\infty$ and $\Phi(0)$ is the largest zero of $\psi$. For real $x\leq c$ recall the classical identity \cite[Eq.~(3.15)]{kyprianou}
\begin{equation}\label{eq:classical}
\PP_x[e^{-q\tau_c^+};\tau_c^+<\infty]=e^{-\Phi(q)(c-x)}.
\end{equation}

Further, for $\lambda\in [0,\infty)$, $\Wm:\mathbb{R}\to [0,\infty)$ will be the $\lambda$-scale function of $X$, characterized by being continuous on $[0,\infty)$, vanishing on $(-\infty,0)$, and having Laplace transform 
\begin{equation}\label{eq:laplace}
\int_0^\infty e^{-\theta x}\Wm(x)dx=\frac{1}{\psi(\theta)-\lambda},\quad \theta\in (\Phi(\lambda),\infty).
\end{equation}
In particular we set $W^{(0)}=:W$. The reader is referred to \cite{kkr} for further background on scale functions of snLp; we note explicitly only the asymptotic behavior \cite[Eq.~(33), Lemmas~2.3 and~3.3]{kkr}
\begin{equation}\label{eq:aux}
e^{-\Phi(\lambda)x}W^{(\lambda)}(x)=W_{\Phi(\lambda)}(x)\uparrow \frac{1}{\psi'(\Phi(\lambda)+)}\text{ as }x\uparrow\infty,\quad \lambda\in [0,\infty),
\end{equation}
that we shall use repeatedly in what follows (here $1/0:=\infty$ when $\lambda=\Phi(0)=\psi'(0+)=0$, and otherwise $\psi'(\Phi(\lambda)+)\in (0,\infty)$; $W_{\Phi(\lambda)}$ is the scale function of an Esscher transformed process -- its precise character is unimportant, what matters is only the monotone convergence).  

Convolution on the real line will be denoted by a $\star$: for Borel measurable $f,g:\mathbb{R}\to \mathbb{R}$, $$(f\star g)(x):=\int_{-\infty}^\infty f(y)g(x-y)dy,\quad x\in \mathbb{R},$$ whenever the Lebesgue integral is well-defined. 

Finally, it will be convenient to introduce the following concepts.
\begin{definition}
For a function $f:\mathbb{R}\to \mathbb{R}$, we will (i) say that it has a \textbf{bounded left tail} (resp. \textbf{left tail that is bounded below away from zero}) if $f$ is bounded (resp. bounded below away from zero) on $(-\infty,x_0]$ for some $x_0\in \mathbb{R}$;   (ii) for further $\alpha\in [0,\infty)$, say that it has a \textbf{left tail that is $\alpha$-subexponential} provided that for some $x_0\in \mathbb{R}$, some $\gamma<\infty$, and then all $x\in (-\infty,x_0]$, one has $\vert f(x)\vert\leq  \gamma e^{\alpha x}$; and (iii) say simply that it has a \textbf{subexponential left tail} if, for some $\alpha>0$, it has a left tail that is $\alpha$-subexponential.
\end{definition}

\section{Results and their proofs}\label{section:results}
Here is now the main result of this note.
\begin{theorem}\label{theorem}
There exists a unique function $\Ho:\mathbb{R}\to (0,\infty)$ satisfying (the arbitrary normalization condition) $\Ho(0)=1$ such that 
\begin{equation}\label{eq:scale}
\Bo(x,c)=\frac{\Ho(x)}{\Ho(c)}\text{ for all real }x\leq c.
\end{equation}
The function $\Ho$ enjoys the following properties.
\begin{enumerate}[(I)]
\item\label{I} It is  nondecreasing (hence locally bounded), continuous, and it is strictly increasing provided $\omega>0$.
\item\label{II} For each $c\in \mathbb{R}$ the following holds: $\omega_1(\cdot \land c)=\omega_2(\cdot\land c)$ implies $\mathcal{H}^{(\omega_1)}(\cdot\land c)=\alpha \mathcal{H}^{(\omega_2)}(\cdot \land c)$ for some $\alpha\in (0,\infty)$; this  $\alpha$ being $1$ if $c\geq 0$.\footnote{$\omega_1(\cdot\land c)$ means the function $(\mathbb{R}\ni x\mapsto \omega_1(x\land c))$. More generally, throughout this text, given an expression $\mathcal{R}(x)$ defined for $x\in R$ we will write $\mathcal{R}(\cdot)$ for the function $(R\ni x\mapsto \mathcal{R}(x))$.}
\item\label{III} If  $\omega_1,\omega_2:\mathbb{R}\to [0,\infty)$ are both locally bounded and Borel measurable with $\omega_1\leq\omega_2$ (resp. $\omega_1<\omega_2$), then  $\mathcal{H}^{(\omega_1)}_q\leq \mathcal{H}^{(\omega_2)}_q$ (resp. $\mathcal{H}^{(\omega_1)}_q<\mathcal{H}^{(\omega_2)}_q$) on $(0,\infty)$ and $\mathcal{H}^{(\omega_1)}_q\geq  \mathcal{H}^{(\omega_2)}_q$ (resp. $\mathcal{H}^{(\omega_1)}_q>\mathcal{H}^{(\omega_2)}_q$) on $(-\infty,0)$ (of course $\mathcal{H}^{(\omega_1)}_q(0)=1=\mathcal{H}^{(\omega_2)}_q(0)$).
\item\label{IV} For real $x\leq c$, $\Ho(x)\leq \Ho(c)e^{-\Phi(q)(c-x)}$, in particular $\Ho(x)\leq e^{\Phi(q)x}$ for all $x\in (-\infty,0]$, so that $\Ho$ has a left tail that is $\Phi(q)$-subexponential.
\end{enumerate}
Furthermore, if $(\omega e^{\Phi(q+p)\cdot})\star \Wq$ is finite-valued for all $p\in (0,\infty)$, then for some unique $\Lo\in [0,1]$, $\Ho$ satisfies the convolution equation
\begin{equation}\label{eq:renewal-zussamen}
\Ho=\Lo e^{\Phi(q)\cdot}+(\omega \Ho)\star \Wq.
\end{equation}
More specifically:
\begin{enumerate}[(i)]
\item\label{conv:i} If moreover $\omega$ has a left tail that is bounded and bounded below away from zero, then $\Ho$ satisfies the (homogeneous) convolution equation 
\begin{equation}\label{eq:renewal}
\Ho=(\omega\Ho)\star \Wq.
\end{equation}
\item\label{conv:ii} 
If even $(\omega e^{\Phi(q)\cdot})\star \Wq$ is finite-valued, in particular if $\omega$ has a subexponential left tail, then $\Ho$ is the unique locally bounded Borel measurable function $H:\mathbb{R}\to \mathbb{R}$ admitting a left tail that is $\Phi(q)$-subexponential and satisfying the (inhomogeneous)  convolution equation 
\begin{equation}\label{eq:renewal:bis}
H=\Lo e^{\Phi(q)\cdot}+(\omega H)\star \Wq,
\end{equation}
where \begin{equation}\label{eq:limit} \Lo=\lim_{x\to-\infty}\Ho(x)e^{-\Phi(q)x}=\lim_{x\to-\infty}\PP_x\left[\exp\left(-\int_0^{\tau_0^+}\omega(X_s)ds\right)\bigg\vert \tau_0^+<e_q\right]\in (0,1].
\end{equation}This function is given as $\Ho\!\!\!=\uparrow\!\!\!\text{--}\!\lim_{n\to \infty}H_n$, where $H_0:=\Lo e^{\Phi(q)\cdot}$ and recursively $H_{n+1}:=\Lo e^{\Phi(q)\cdot}+(\omega H_n)\star \Wq$ for $n\in  \mathbb{N}_0$. 
\end{enumerate}
\end{theorem}
After some remarks and examples we turn to the proof of this theorem on p.~\pageref{proof}.
\begin{remark}
Since $\Ho=\mathcal{H}_0^{(\omega+q)}$, \eqref{eq:renewal} may be rewritten as $\Ho=((\omega+q) \Ho)\star W$. For the same reason, when $q>0$, then automatically $\Ho=((\omega+q) \Ho)\star W$.
\end{remark}
\begin{remark}
Because of \eqref{eq:aux} cases \ref{conv:i} and \ref{conv:ii} are seen to be mutually exclusive (but they are not exhaustive). Of course  $(\omega e^{\Phi(q+p)\cdot})\star \Wq$ is finite-valued for all $p\in (0,\infty)$ iff $(\omega e^{\alpha \cdot}e^{\Phi(q)\cdot})\star \Wq$ is finite-valued for all $\alpha\in (0,\infty)$, in which case, for each $\alpha\in (0,\infty)$, $\omega e^{\alpha\cdot}$ falls under the provisos of \ref{conv:ii}. For the resulting convolution equation \eqref{eq:renewal:bis} we then have suitable uniqueness of the solution as well as an explicit recursion to (at least in principle) produce it. At the same time, by bounded convergence in \eqref{eq:scale}-\eqref{what-we-are-after}, $\lim_{\alpha\downarrow 0}\mathcal{H}^{(\omega e^{\alpha\cdot})}_q=\Ho$.
\end{remark}
\begin{example}\label{example:one}
When $\omega$ is constant and equal to some $\mu\in [0,\infty)$, then \eqref{eq:classical} $\Ho=\mathcal{H}^{(\mu)}_q=e^{\Phi(q+\mu)\cdot}$, and this case falls under \ref{conv:i} or \ref{conv:ii}, according as $\mu>0$ or $\mu=0$.
\end{example}
\begin{example}\label{example:two}
When $\omega=\gamma e^{\alpha\cdot}$, with $\gamma\in [0,\infty)$ and $\alpha \in (0,\infty)$, a case that falls under \ref{conv:ii}, one obtains using \eqref{eq:laplace}
\begin{equation}\label{eq:pssMp-exit}
\mathcal{H}^{(\gamma e^{\alpha\cdot })}_q(x)= \sum_{k=0}^\infty\frac{\gamma^k e^{(\Phi(q)+\alpha k)x}}{\prod_{l=1}^k(\psi(\Phi(q)+l\alpha)-q)}\bigg/ \sum_{k=0}^\infty\frac{\gamma^k}{\prod_{l=1}^k(\psi(\Phi(q)+l\alpha)-q)},\quad x\in \mathbb{R},
\end{equation}
with the series converging to finite values. (As usual the empty product is interpreted as being equal to $1$.) Of course when $\gamma>0$, then from \eqref{eq:scale}, by spatial homogeneity, $\mathcal{H}^{(\gamma e^{\alpha\cdot})}_q(x)=\frac{\mathcal{H}^{(e^{\alpha \cdot})}_q(x+\frac{1}{\alpha}\log \gamma)}{\mathcal{H}^{(e^{\alpha \cdot})}_q(\frac{1}{\alpha}\log \gamma)}$, $x\in \mathbb{R}$. Note that this reproduces (up to trivial transformations) Patie's scale functions from the fluctuation theory of spectrally negative pssMp \cite[Section~13.7]{kyprianou}. One also identifies the limit  \eqref{eq:limit} as $L^{(\gamma e^{\alpha\cdot})}_q=\left(\sum_{k=0}^\infty\frac{\gamma^k}{\prod_{l=1}^k(\psi(\Phi(q)+l\alpha)-q)}\right)^{-1}$. 
\end{example}
\begin{remark}
Let $\gamma\in [0,\infty)$ and $\alpha \in (0,\infty)$. Suppose $\omega(x)\leq \gamma e^{\alpha x}$ for all $x\in \mathbb{R}$. Then, again via \eqref{eq:laplace}, one gets the following a priori bound on the absolute error in \ref{conv:ii} from computing only finitely many terms of the recursion for $\Ho$: $\Ho-H_n\leq \Lo\sum_{k=n+1}^\infty\frac{\gamma^k}{\prod_{l=1}^k(\psi(\Phi(q)+\alpha l)-q)}e^{(\Phi(q)+\alpha k)\cdot}\leq \sum_{k=n+1}^\infty\frac{\gamma^k}{\prod_{l=1}^k(\psi(\Phi(q)+\alpha l)-q)}e^{(\Phi(q)+\alpha k)\cdot}$ for all $n\in \mathbb{N}_0$. (In particular $\Ho\leq \Lo\sum_{k=0}^\infty\frac{\gamma^k}{\prod_{l=1}^k(\psi(\Phi(q)+\alpha l)-q)}e^{(\Phi(q)+\alpha k)\cdot}\leq \sum_{k=0}^\infty\frac{\gamma^k}{\prod_{l=1}^k(\psi(\Phi(q)+\alpha l)-q)}e^{(\Phi(q)+\alpha k)\cdot}$.)
\end{remark}
\begin{example}\label{example:csbp}
Let $c\in (0,\infty)$, $\gamma\in (0,\infty)$ and $\omega(x)=\frac{\gamma}{\vert x\vert}$ for $x\in (-\infty,-c]$. Using the result for csbp of \cite[Theorem~1]{duhalde}  we identify $\Ho$ up to a multiplicative constant; $$\Ho(x)\propto \int_{\Phi(q)}^\infty\frac{dz}{\psi(z)-q}\exp\left(xz+\int_\theta^z\frac{\gamma}{\psi(u)-q}du\right),\quad x\in (-\infty,-c],$$ where $\theta\in (\Phi(q),\infty)$ is arbitrary but fixed. Note that this $\omega$ falls neither under \ref{conv:i} nor under \ref{conv:ii}, but it does fall under \eqref{eq:renewal-zussamen}. In fact, while it is not so obvious, an easy computation shows that \eqref{eq:renewal-zussamen} is verified in this case with $\Lo=0$.
\end{example}

\begin{example}\label{example:new}
Let $n\in \mathbb{N}_{\geq 2}$, $\gamma\in [0,\infty)$, $c\in (0,\infty)$ and $\omega(x)=\frac{\gamma}{\vert x\vert^n}$ for $x\in (-\infty,-c]$. Except possibly when $q=\Phi(0)=\psi'(0+)=0$, we then automatically have, because of the asymptotic properties of $\Wq$, see \eqref{eq:aux}, that $(\omega e^{\Phi(q)\cdot})\star \Wq$ is finite-valued, and in any event we assume now that this is so. Then note, using \eqref{eq:laplace}, that, for $x\in (-\infty,-c]$, $v\in [\Phi(q),\infty)$ and for $ \alpha>0$, $\frac{d^n}{d\alpha^n}\int_0^\infty \frac{e^{(v+\alpha)(x-y)}}{(x-y)^n}\Wq(y)dy=\frac{e^{(v+\alpha)x}}{\psi(v+\alpha)-q}$, which implies $\int_0^\infty \frac{e^{(v+\alpha)(x-y)}}{\vert x-y\vert ^n}\Wq(y)dy=\int_{v+\alpha}^\infty dv_1\int_{v_1}^\infty dv_2\cdots \int_{v_{n-1}}^\infty dv_n\frac{e^{v_nx}}{\psi(v_n)-q}=\int_{v+\alpha}^\infty dv_n \frac{e^{xv_n}}{\psi(v_n)-q}\int_{v+\alpha}^{v_n}dv_{n-1}\cdots\int_{v+\alpha}^{v_2}dv_1=\int_{v+\alpha}^\infty dy\frac{e^{xy}}{\psi(y)-q}\frac{(y-v-\alpha)^{n-1}}{(n-1)!}$. Hence, letting $\alpha\downarrow 0$, by monotone convergence, $((\omega e^{v\cdot})\star \Wq)(x)=\gamma \int_0^\infty \frac{e^{v(x-y)}}{\vert x-y\vert^n}\Wq(y)dy=\gamma\int_{v}^\infty dy\frac{e^{xy}}{\psi(y)-q}\frac{(y-v)^{n-1}}{(n-1)!}=\gamma\int_0^\infty dy\frac{e^{x(v+y)}}{\psi(v+y)-q}\frac{y^{n-1}}{(n-1)!}$. Thus the recursion of \ref{conv:ii} allows us to identify $\Ho$, up to a proportionality constant, as an infinite series of iterated integrals; $$\Ho(x)=e^{\Phi(q)x}\bigg[1+\frac{\gamma}{(n-1)!}\int_0^\infty dy\frac{y^{n-1}e^{xy}}{\psi(\Phi(q)+y)-q}$$ $$+\left(\frac{\gamma}{(n-1)!}\right)^2\int_0^\infty dy\frac{y^{n-1}e^{xy}}{\psi(\Phi(q)+y)-q}\int_0^\infty dz\frac{z^{n-1}e^{xz}}{\psi(\Phi(q)+y+z)-q}+\cdots\bigg],\quad x\in (-\infty,-c].$$
\end{example}

\begin{remark}\label{remark}
In connection to the results of \cite{zbigniew-bo}:
\begin{enumerate}[(a)]
\item Let $q=0$. If $\omega\vert_{(-\infty,0]}= 0$, then $\Lo=1$ and \eqref{eq:renewal:bis} recovers \cite[Eq.~(2.23) with $\phi=0$]{zbigniew-bo}. If $\omega\vert_{(-\infty,0]}= \phi\in (0,\infty)$, then  $\Ho\vert_{(-\infty,0]}=e^{\Phi(\phi)\cdot}$ and \eqref{eq:renewal} is seen (via  $W^{(\phi)}=W+\phi W^{(\phi)}\star W$, that may be checked by taking Laplace transforms using \eqref{eq:laplace};  and via $\phi e^{\Phi(\phi)\cdot}\star W=e^{\Phi(\phi)\cdot}$, which follows directly from \eqref{eq:laplace}) to be a slight rewriting of \cite[Eq.~(2.23) with $\phi>0$]{zbigniew-bo}.
\item In \cite[Eq.~(2.25)]{zbigniew-bo}, for real $x\leq c$, the ``$\omega$-resolvent'' identity  $\int_0^\infty\PP_x\left[\exp\left(-\int_0^t\omega(X_u)du\right);t<\tau_c^+,X_t\in dy\right]dt$ is formally asserted only for the case when $\omega\vert_{(-\infty,0]}$ is constant, but it prevails of course in full generality (with our $\mathcal{H}^{(\omega)}_0$ replacing the $\mathcal{H}^{(\omega)}$ there): the proof consists only in using the resolvent identity for the two-sided exit problem \cite[Eq.~(2.15)]{zbigniew-bo} and the fact that  $\frac{\mathcal{H}^{(\omega)}_0(x)}{\mathcal{H}^{(\omega)}_0(c)}=\mathcal{B}_0^{(\omega)}(x,c)=\lim_{b\to -\infty}\PP_x\left[\exp\left\{-\int_0^{\tau_c^+}\omega(X_u)du\right\};\tau_c^+<\tau_b^-\right]$. For this reason we omit reproducing the expression here. 
\item The approach of \cite{zbigniew-bo} to handle the case when $\omega$ is constant on $(-\infty,0]$ is by taking limits  in the two-sided exit problem (as indicated in the previous item). We will follow an alternate, more direct route (also inspired by \cite{zbigniew-bo}), which will allow us to prove the result in greater generality.
\end{enumerate}
\end{remark}
\begin{proof}[Proof of Theorem~\ref{theorem}]\label{proof}
Let $x\leq y\leq c$ be real numbers. Since $X$ has no positive jumps, then $\PP_x$-a.s. $X_{\tau_y^+}=y$ on $\{\tau_y^+<\infty\}$, and  it follows by the strong Markov property of $X$ applied at the time $\tau_y^+$ and the memoryless property of the exponential distribution, that one has the multiplicative structure $$\Bo(x,c)=\Bo(x,y)\Bo(y,c).$$ Furthermore, it is clear that $\Bo(x,c)>0$ for all real $x\leq c$. As a consequence we may unambiguously define (with a preemptive choice of notation) $\Ho(x):=\frac{\Bo(x,c)}{\Bo(0,c)}$ for real $x\leq c$, $c\geq 0$. In short, then, $\Ho:\mathbb{R}\to (0,\infty)$, $\Ho(0)=1$, and \eqref{eq:scale} holds. It is clear that $\Ho$ is unique in having the preceding properties. 

Statements \ref{I}-\ref{II}-\ref{III}, apart from the continuity of $\Ho$, follow immediately from \eqref{eq:scale}-\eqref{what-we-are-after} and simple comparison arguments. To prove continuity of $\Ho$ note that for real $c$, $\tau_d^+\downarrow \tau_c^+$ as $d\downarrow c$, and that further for $x\in (-\infty,c)$, by  quasi-left-continuity and regularity of $0$ for $(0,\infty)$, $\PP_x$-a.s. also $\tau_d^+\uparrow \tau_c^+$ (on $\{\lim_{d\uparrow c}\tau_d^+<\infty\}$ and hence everywhere) as $d\uparrow c$. Then use bounded convergence in \eqref{eq:scale}-\eqref{what-we-are-after}, exploiting the facts that $\omega$ is locally bounded and that the law of the overall supremum $\overline{X}_\infty$ has no finite atoms, which implies that a.s.-$\PP_x$ on $\{\tau_c^+=\infty\}$ also $\tau_d^+=\infty$ for all $d<c$ that are sufficiently close to $c$.  (That the law of $\overline{X}_\infty$ has no finite atoms follows for instance from \eqref{eq:classical} and the continuity of $\Phi$.) For \ref{IV}, notice that by \eqref{eq:classical},  $\Ho(x)\leq\Ho(c)\PP_x(\tau_c^+<e_q)=\Ho(c)\PP_x[e^{-q\tau_c^+};\tau_c^+<\infty]=\Ho(c)e^{-\Phi(q)(c-x)}$ for real $x\leq c$.

We now prove \ref{conv:i}. Since the homogeneous convolution equation \eqref{eq:renewal} may be checked ``locally'', separately on each $(-\infty,c ]$ for $c\in \mathbb{R}$, we may assume (replacing $\omega$ by $\omega(\cdot\land c)$ if necessary) that $\omega$ is bounded by a $\lambda\in (0,\infty)$. Also, there is an $x_0\in \mathbb{R}$ such that $\omega$ is bounded below away from zero on $(-\infty,x_0]$ by some $a>0$. Let again $x\leq c$ be real numbers. Then, using the resolvent \cite[Theorem~2.7(ii)]{kkr} \footnotesize $$\int_0^\infty e^{-\lambda t}\PP_x(X_t\in dy,t<\tau_c^+\land e_q)dt=\left(e^{-\Phi(\lambda+q)(c-x)}\Wl(c-y)-\Wl(x-y) \right)dy\text{ for } y\in (-\infty,c],$$ \normalsize as well as the classical identity \eqref{eq:classical} $\PP_x[e^{-\lambda \tau_c^+};\tau_c^+<e_q]=e^{-\Phi(\lambda+q)(c-x)}$, it follows via the marked Poisson process technique of \cite{zbigniew-bo} -- letting $(T_i)_{i\in \mathbb{N}}$ be the arrival times of a homogeneous Poisson process of intensity $\lambda$, marked by an independent sequence $(M_i)_{i\in \mathbb{N}}$ of i.i.-uniformly on $[0,\lambda]$-d. random variables -- that  
$$\Bo(x,c)=\PP_x(M_i>\omega(X_{T_i})\text{ for all $i\in \mathbb{N}$ such that }T_i<\tau_c^+,\tau_c^+<e_q)$$
$$=\PP_x(T_1>\tau_c^+,\tau_c^+<e_q)+\PP_x[\Bo(X_{T_1},c);T_1<\tau_c^+\land e_q,M_1>\omega(X_{T_1})]$$
\begin{equation}\label{eq:HPP-trick}
=e^{-\Phi(\lambda+q)(c-x)}+\int_{-\infty}^c\Bo(y,c)(e^{-\Phi(\lambda+q)(c-x)}\Wl(c-y)-\Wl(x-y))(\lambda-\omega(y))dy.
\end{equation}
Next, plugging  \eqref{eq:scale} into \eqref{eq:HPP-trick}, multiplying both sides by $\Ho(c)$ and letting $c\uparrow \infty$, we obtain by monotone convergence using \eqref{eq:aux} that

$$\Ho=e^{\Phi(\lambda+q)\Id}h_\lambda+\left((\lambda-\omega)\Ho\right)\star \left(\frac{e^{\Phi(\lambda+q)\Id}}{\psi'(\Phi(\lambda+q))}-\Wl\right)$$ with $h_\lambda:=\lim_{c\to\infty}\Ho(c)e^{-\Phi(\lambda+q)c}$; a priori this limit  must  exist in $[0,\infty)$. Now convolute both sides of the preceding display by $\lambda \Wq$, exploiting the relations $e^{\Phi(\lambda+q)\cdot}=\lambda e^{\Phi(\lambda+q)\cdot}\star W^{(q)}$ (which is a direct consequence of \eqref{eq:laplace}) and $ \Wl=W^{(q)}+\lambda \Wl\star W^{(q)}$ (which may be checked by taking Laplace transforms and again using \eqref{eq:laplace}) that together imply $\left(\frac{e^{\Phi(\lambda+q)\Id}}{\psi'(\Phi(\lambda+q))}-\Wl\right)\star \lambda W^{(q)}=\left(\frac{e^{\Phi(\lambda+q)\Id}}{\psi'(\Phi(\lambda+q))}-\Wl\right)+W^{(q)}$, to obtain\footnotesize
$$\lambda \Ho\star \Wq=e^{\Phi(\lambda+q)\Id}h_\lambda+\left((\lambda-\omega)\Ho\right)\star \left(\frac{e^{\Phi(\lambda)\Id}}{\psi'(\Phi(q))}-\Wl+\Wq\right)=\Ho+\left((\lambda-\omega)\Ho\right)\star \Wq.$$\normalsize
Then the estimate $\frac{\Ho(x)}{\Ho(x_0)}\leq e^{-\Phi(q+a)(x_0-x)}$ for $x\in (-\infty, x_0]$ implies (via \eqref{eq:laplace} and the local boundedness of $\Ho$ and $\Wq$) that $\Ho\star \Wq$ is finite-valued and upon subtracting finite quantities we obtain \eqref{eq:renewal}. This concludes the proof of \ref{conv:i}.

Suppose now $(\omega e^{\Phi(q+p)\cdot})\star \Wq$ is finite-valued for all $p\in (0,\infty)$. From \ref{conv:i}, for each $p\in (0,\infty)$ and $n\in \mathbb{N}$, one has, for all $x\in \mathbb{R}$,
$$\mathcal{H}^{(p+\omega\land n)}_q(x)=[((\omega\land n+p)\mathcal{H}^{(p+\omega\land n)}_q)\star \Wq](x)=[((\omega\land n) \mathcal{H}^{(p+\omega\land n)}_q)\star \Wq](x) +p[\mathcal{H}^{(p+\omega\land n)}_q\star \Wq](x)
$$ \footnotesize
\begin{equation}\label{eq:on-the-way}
=\left[\int_{-\infty}^{0\land x}+\int_0^{0\lor x}\right](\omega(y)\land n)\mathcal{H}^{(p+\omega\land n)}_q(y)\Wq(x-y)dy+p\left[\int_{-\infty}^{0\land x}+\int_0^{x\lor 0}\right]\mathcal{H}^{(p+\omega\land n)}_q(y)\Wq(x-y)dy.
\end{equation} 
\normalsize
We now first pass to the limit  $n\to\infty$ as follows. In \eqref{eq:scale}-\eqref{what-we-are-after}  monotone (for the integral against the Lebesgue measure) and bounded (for the expectation) convergence yield $\mathcal{H}^{(p+\omega\land n)}_q\to \mathcal{H}^{(p+\omega)}_q$ as $n\to\infty$. Then in \eqref{eq:on-the-way} monotone (for the integrals on $[0,x\lor 0]$; recall \ref{III}) and dominated (for the integrals on $(-\infty,0\land x]$; using the assumed integrability condition and the estimate $\mathcal{H}^{(p+\omega\land n)}_q(y)\leq e^{\Phi(q+p)y}$ for $y\in (-\infty,0]$) convergence produce
$$\mathcal{H}^{(p+\omega)}_q(x)=[(\omega \mathcal{H}^{(p+\omega)}_q)\star \Wq](x) +p[\mathcal{H}^{(p+\omega)}_q\star \Wq](x)
$$ 
\begin{equation}\label{eq:on-the-way-2}
=\left[\int_{-\infty}^{0\land x}+\int_0^{0\lor x}\right]\omega(y)\mathcal{H}^{(p+\omega)}_q(y)\Wq(x-y)dy+p\int_{-\infty}^{x}\mathcal{H}^{(p+\omega)}_q(y)\Wq(x-y)dy.
\end{equation} 
Let us next write, for the purposes of the remainder of this proof only, $\mathcal{H}^{(p+\omega)}_q=:H_p$ and $\Ho=:H$ for short. We proceed to pass to the limit $p\downarrow 0$. In \eqref{eq:scale}-\eqref{what-we-are-after}, similarly as before, by bounded  (for the integral against the Lebesgue measure; recall $\omega$ is locally bounded) and monotone (for the expectation) convergence, we obtain that $H_p\to H$ as $p\downarrow 0$. Then in \eqref{eq:on-the-way-2}, $[(\omega H_p)\star \Wq](x)\to [(\omega H)\star \Wq](x)$, as $p\downarrow 0$, by monotone (for the integral on $(-\infty,0\land x]$; recall \ref{III}) and bounded (for the integral on $[0,x\lor 0]$; use the facts that $H_p$ is nondecreasing, that $H_p(c)$ is bounded in bounded $p$ given a fixed $c\in [0,\infty)$, and that $\Wq$ and $\omega$ are locally bounded) convergence.  Finally we consider $$L_x:=\lim_{p\downarrow 0}e^{-\Phi(q)x}p(H_p\star \Wq)(x)=e^{-\Phi(q)x}\lim_{p\downarrow 0}p\int_{-\infty}^xH_p(y)\Wq(x-y)dy;$$ a priori this limit must exist in $[0,\infty)$. We show that $L_x$ does not depend on $x$, thus demonstrating that \eqref{eq:renewal-zussamen} is indeed satisfied for some, necessarily unique, $\Lo\in [0,1]$. 

Now, since $\Wq$ is locally bounded, since $H_p$ is nondecreasing, and since $H_p(c)$ is bounded in bounded $p$ given a fixed real $c$, it is clear that for any choice of $a\in (-\infty,x]$, $$L_x=\lim_{p\downarrow 0}p\int_{-\infty}^aH_p(y)e^{-\Phi(q)y}\Wq(x-y)e^{-\Phi(q)(x-y)}dy.$$ Suppose now first that $\psi'(\Phi(q)+)>0$. Then, given any $\epsilon>0$ we may \eqref{eq:aux} choose this $a$ to be (for simplicity) $\leq 0$ and such as to render $\vert \Wq(x-y)e^{-\Phi(q)(x-y)}-1/\psi'(\Phi(q)+)\vert\leq \epsilon$ for all $y\leq a$. Consequently, since (using the estimate $H_p(y)\leq e^{\Phi(q+p)y}$ for $y\leq 0$) $\limsup_{p\downarrow 0}p\int_{-\infty}^0H_p(y)e^{-\Phi(q)y}dy\leq \lim_{p\downarrow 0}\frac{p}{\Phi(q+p)-\Phi(q)}=\psi'(\Phi(q)+)<\infty$, we conclude that $L_x$ in fact does not depend on $x$. For the case when $\psi'(\Phi(q)+)=0$, i.e. the case $q=\Phi(q)=\psi'(0+)=0$,  we have that $L_x=\lim_{p\downarrow 0}p\int_{-\infty}^aH_p(y)W(x-y)dy$ for any $a\in (-\infty,x]$. We argue that $Q:=\limsup_{p\downarrow 0}p\int_{-\infty}^aH_p(y)(W(x-y)-W(a-y))dy=0$ for  $a$ that are (again for simplicity) $\leq 0\land x$, which will complete the verification that $L_x$ does not depend on $x$. Indeed, since $H_p(y)\leq H_p(a)e^{-\Phi(p)(a-y)}\leq e^{-\Phi(p)(a-y)}$ for $y\leq a\leq 0$, we have that $Q\leq \lim_{p\downarrow 0}p\int_{-a}^\infty e^{-\Phi(p)(y+a)}(W(x+y)-W(a+y))dy=\lim_{p\downarrow 0}p[e^{\Phi(p)(x-a)}(p^{-1}-\int_0^{x-a}e^{-\Phi(p)z}W(z)dz)-p^{-1}]=0$. 

The claims of \ref{conv:ii} follow at once from the above and from Lemma~\ref{lemma} to feature immediately.
\end{proof}
We have, regarding uniqueness of the solutions to \eqref{eq:renewal:bis}, the following

\begin{lemma}\label{lemma}
Suppose $(\omega e^{\Phi(q)\cdot})\star \Wq$ is finite-valued (which obtains if $\omega$ has a subexponential left tail). Let $G:\mathbb{R}\to \mathbb{R}$ be Borel measurable and locally bounded with a left tail that is $\Phi(q)$-subexponential. Then:
\begin{enumerate}[(i)]
\item \label{lemma:o} $\lim_{x\to -\infty}e^{-\Phi(q)x}((\omega G)\star \Wq)(x)=0$.
\item\label{lemma:i} $G=(G\omega)\star \Wq$ implies $G= 0$. 
\item\label{lemma:ii} Let now further $g:\mathbb{R}\to [0,\infty)$ be locally bounded Borel measurable with a left tail that is $\Phi(q)$-subexponential. Suppose $G\geq 0$ and $G=g+(\omega G)\star \Wq$. Then $G=G_\infty:={\uparrow\!\!\!\text{--}\!\lim}_{n\to\infty}G_n$, where the $G_n$ are given recursively: $G_0:=g$ and $G_{n+1}:=g+(\omega G_n)\star \Wq$ for $n\in \mathbb{N}_0$.
\end{enumerate}
\end{lemma}
\begin{proof}
\ref{lemma:o}. We have $\vert e^{-\Phi(q)x}((\omega G)\star \Wq)(x)\vert \leq  \int_{-\infty}^\infty\omega(y)\vert G(y)\vert e^{-\Phi(q)y}\Wq(x-y)e^{-\Phi(q)(x-y)}dy$. Since $G$ has a left tail that is $\Phi(q)$-subexponential and since it is locally bounded it follows that there is a $\gamma<\infty$ such that $\vert G(y)\vert e^{-\Phi(q)y}\leq \gamma$ for all $y\in (-\infty,0]$ (say). Therefore, for $x\in (-\infty,0]$, $\vert e^{-\Phi(q)x}((\omega G)\star \Wq)(x)\vert \leq \gamma\int_{-\infty}^\infty \omega(y)\Wq(x-y)e^{-\Phi(q)(x-y)}dy$, which is $<\infty$ by assumption. Now by \eqref{eq:aux} $\Wq(x-y)e^{-\Phi(q)(x-y)}$ is nonincreasing to $0$ as $x\downarrow-\infty$. Thus the conclusion follows by dominated convergence.

\ref{lemma:i}. Denote, for $x\in \mathbb{R}$,  $\Vert G\Vert_{x}:=\sup_{y\in (-\infty,x]}\vert G(y)\vert e^{-\Phi(q)y}$. Note this quantity is finite because $G$ has a tail that is $\Phi(q)$-subexponential and because it is locally bounded. Then $G=(G\omega)\star \Wq$ implies that for all $x\in \mathbb{R}$ one has $\vert G(x)\vert e^{-\Phi(q)x}\leq  \Vert G\Vert_x\int_{-\infty}^\infty\omega(y)\Wq(x-y)e^{-\Phi(q)(x-y)}dy= \Vert G\Vert_x (\omega\star (e^{-\Phi(q)\cdot }\Wq))(x)$. By \ref{lemma:o} $ (\omega\star (e^{-\Phi(q)\cdot }\Wq))(x)=e^{-\Phi(q)x}((\omega e^{\Phi(q)\cdot})\star \Wq)(x)\to 0 $ as $x\downarrow -\infty$, so there is an $x_0\in \mathbb{R}$ such that $(\omega\star (e^{-\Phi(q)\cdot }\Wq))(x)\leq I_{x_0}:=(\omega\star (e^{-\Phi(q)\cdot }\Wq))(x_0)<1$ for all $x\in (-\infty,x_0]$. At the same time, the above estimate implies $\Vert G\Vert_{x_0}\leq I_{x_0}\Vert G\Vert_{x_0}$, hence $\Vert G\Vert_{x_0}=0$, which forces $G$ to vanish on $(-\infty,x_0]$. Let us now shift all the functions by $x_0$ for (notational) convenience; to wit $F:=G(x_0+\cdot)$ and $\theta:=\omega(x_0+\cdot)$ are Borel measurable, locally bounded and $F=(F\theta)\star \Wq$.  From this we obtain finally that $F= 0$ by the following argument. Fix $y_0\in [0,\infty)$ and let $\theta_0$ be an upper bound for $\theta$ on $[0,y_0]$. We may choose $s_0\in (0,\infty)$ such that $\psi(s_0)-q>\theta_0$. Denote $\Vert F\Vert:=\sup_{y\in [0,y_0]}\vert F(y)\vert e^{-s_0y}$. Then, for each $y\in [0,y_0]$, $F=(F\theta)\star \Wq$ and \eqref{eq:laplace} imply $\vert F(y)\vert e^{-s_0y}\leq \Vert F\Vert \theta_0\int_0^{y}e^{-s_0(y-z)}\Wq(y-z)dz\leq  \Vert F\Vert \theta_0/(\psi(s_0)-q)$. Again this renders $\Vert F\Vert=0$, and completes the proof of \ref{lemma:i}. 

\ref{lemma:ii}. By induction one proves that $G_n\leq G$ for all $n\in \mathbb{N}_0$. Moreover, passing to the limit in the recursion via monotone convergence, one finds that $G_\infty=g+(\omega G_\infty)\star \Wq$. It follows that $G-G_\infty$ has a $\Phi(q)$-subexponential left tail, is locally bounded, Borel measurable and satisfies $G-G_\infty=((G-G_\infty)\omega)\star\Wq$. By \ref{lemma:i} it means that $G=G_\infty$.
\end{proof}
As is to be expected, the solution to \eqref{what-we-are-after} is associated to a family of (local) martingales. Recall the process $Y$ from \eqref{eq:Y}.

\begin{proposition}\label{proposition:martingale}
Let $c\in \mathbb{R}$, $\gamma\in (0,\infty)$. Define the processes $Z$ and $W$ as follows:
 $$Z_t:=\exp\left(-\int_0^t\omega(X_u)du-qt\right)\Ho(X_t),\quad t\in [0,\infty),$$
and
$$W_s:=e^{-\gamma s}\mathcal{H}^{(\gamma\omega)}_q(Y_s)\mathbbm{1}_{\{s<\zeta\}},\quad s\in [0,\infty).$$
Let further $\mathcal{F}=(\FF_t)_{t\in [0,\infty)}$ be any filtration relative to which $X$ is adapted and has independent increments. Then:
\begin{enumerate}[(i)]
\item\label{mtg:i} The stopped process $Z^{\tau_c^+}$ is a bounded c\`adl\`ag martingale in the filtration $\mathcal{F}$ under $\PP_x$ for each $x\in \mathbb{R}$; for real $x\leq c$ the $\PP_x$-terminal value of this martingale is $\Ho(c)\exp\{-\int_0^{\tau_c^+}\omega(X_s)ds-q\tau_c^+\}\mathbbm{1}_{\{\tau_c^+<\infty\}}$.
\item\label{mtg:ii} Assume $\omega$ is strictly positive and $e_q$ is independent of $\FF_\infty$ (under $\PP_x$ for each $x\in \mathbb{R}$). Then the stopped process $W^{T_c^+}$ is a c\`adl\`ag bounded martingale in the filtration $\GG=(\GG_s)_{s\in [0,\infty)}:=(\FF_{\rho_s}\lor \sigma(\{\{\rho_u<e_q\}:u\in [0,s]\}))_{s\in [0,\infty)}$ under $\PP_x$ for each $x\in \mathbb{R}$.
\end{enumerate}
\end{proposition}
\begin{remark}
As a check, since $W^{T_c^+}$ has a constant expectation,  we recover \eqref{eq:pssMp-Laplace} in the limit as time goes to infinity.
\end{remark}
\begin{proof} 
We may assume $x\leq c$. 

\ref{mtg:i}. Let $t\in [0,\infty)$. Then in Markov process theory parlance (for notational simplicity only; ultimately no shift operators are of course needed here)  $\tau_c^+\land t+\tau_c^+\circ \theta_{\tau_c^+\land t}=\tau_c^+$ and $\PP_x$-a.s. $Z_{t\land \tau_c^+}=\Ho(c)\exp(-\int_0^{t\land\tau_c^+}\omega(X_u)du-q(t\land\tau_c^+))\PP_{X_{t\land \tau_c^+}}\left[\exp\{-\int_0^{\tau_c^+}\omega(X_u)du-q\tau_c^+\};\tau_c^+<\infty\right]=\Ho(c)\PP_x[(\exp\{-\int_0^{\tau_c^+}\omega(X_u)du-q\tau_c^+\}\mathbbm{1}_{\{\tau_c^+<\infty\}})\circ \theta_{\tau_c^+\land t}\exp\{-\int_0^{t\land\tau_c^+}\omega(X_u)du-q(t\land\tau_c^+)\}\vert \FF_{t\land \tau_c^+}]=\Ho(c)\PP_x[\exp\{-\int_0^{\tau_c^+}\omega(X_u)du-q\tau_c^+\}\mathbbm{1}_{\{\tau_c^+<\infty\}}\vert\FF_{t\land \tau_c^+}]=\Ho(c)\PP_x[\exp\{-\int_0^{\tau_c^+}\omega(X_u)du-q\tau_c^+\}\mathbbm{1}_{\{\tau_c^+<\infty\}}\vert\FF_{t}]$, which establishes the first claim. 

\ref{mtg:ii}. % By \ref{mtg:i}, since $e_q$ is independent of $\FF_\infty$, $Z^{\tau_c^+}$ is a martingale also in the filtration $\FF\lor \sigma(e_q)$. Then, f
For all real $0\leq s\leq t$, $A\in \FF_{\rho_s}$, applying the optional sampling theorem to the process $Z^{\tau_c^+}$ at the times $\rho_s$ and $\rho_t$, we obtain $$\PP_x\left[\exp\left(-\int_0^{\rho_t\land \tau_c^+}\omega(X_u)du-q(\rho_t\land \tau_c^+)\right)\Ho(X_{\rho_t\land\tau_c^+});A\cap\{\rho_t\land \tau_c^+<\infty\}\right]$$
$$=\PP_x\left[\exp\left(-\int_0^{\rho_s\land \tau_c^+}\omega(X_u)du-q(\rho_s\land \tau_c^+)\right)\Ho(X_{\rho_s\land\tau_c^+});A\cap \{\rho_s\land \tau_c^+<\infty\}\right],$$ i.e., because $\rho_t\land \tau_c^+=\rho_t\land \rho_{T_c^+}=\rho_{t\land T_c^+}$ on $\{\rho_t\land\tau_c^+<e_q\}=\{t\land T_c^+<\zeta\}$, since $\rho$ is the inverse of $\int_0^\cdot \omega(X_u)du$, and by the independence of $e_q$ from $\FF_\infty$,\footnotesize
 $$\PP_x\left[\exp\left(-t\land T_c^+\right)\Ho(Y_{t\land T_c^+});A\cap\{t\land T_c^+<\zeta\}\right]=\PP_x\left[\exp\left(-s\land T_c^+\right)\Ho(Y_{s\land T_c^+});A\cap\{s\land T_c^+<\zeta\}\right].$$\normalsize This implies that $(e^{- s}\mathcal{H}^{(\omega)}_q(Y_s)\mathbbm{1}_{\{s<\zeta\}})_{s\in [0,\infty)}$ stopped at $T_c$ is a martingale in the filtration $\GG$ under $\PP_x$, because this process is constant on $[\zeta,\infty)$, and since for $s\in [0,\infty)$, $\{s<\zeta\}=\{\rho_s<e_q\}$ with the equality of the trace $\sigma$-fields  $\mathcal{G}_s\vert_{\{\rho_s<e_q\}}=\mathcal{F}_{\rho_s}\vert_{\{\rho_s<e_q\}}$ holding true. Replacing $\omega$ with $\gamma\omega$ shows the same is true of the process $(e^{-\gamma s}\mathcal{H}^{(\gamma\omega)}_q(Y_s)\mathbbm{1}_{\{s<\zeta\}})_{s\in [0,\infty)}$ stopped at $T_c^+$. 
\end{proof}
Apart from the solutions presented in Examples~\ref{example:one},~\ref{example:two},~\ref{example:csbp} and~\ref{example:new}, it seems difficult to come up with ``nice'' $\omega$ for which $\Ho$ is given explicitly (in terms of $\psi$ and $\Phi$, say), at least for a general $\Wq$. However, based on Example~\ref{example:two},  we can ``reverse-engineer'' a class of $\omega$ for which $\Ho$ is explicit, in the precise sense of

\begin{proposition}
Let $\nu$ be a probability measure on the Borel sets of $(0,\infty)$ whose support is compactly contained in $(0,\infty)$. Denote, for $\alpha\in (0,\infty)$,  $\Ha:=\mathcal{H}^{(e^{\alpha\cdot})}_q$ and $\La:=L^{(e^{\alpha\cdot})}_q$. Then $H=\Ho$ and $L=\Lo$, where $L:=\int \La\nu(d\alpha)$ and, for $x\in \mathbb{R}$,  $H(x):=\int \Ha(x)\nu(d\alpha)$ and $\omega(x):=\frac{\int \Ha(x)e^{\alpha x}\nu(d\alpha)}{\int \Ha(x)\nu(d\alpha)}$.
\end{proposition}
\begin{proof}
The fact that $\nu$ has a bounded support ensures that $H$ is locally bounded. We know $\Ha=\La e^{\Phi(q)\cdot}+(e^{\alpha \cdot}\Ha)\star\Wq$ for each $\alpha\in (0,\infty)$. Integrating both sides against $\nu(d\alpha)$ we obtain via Tonelli's theorem (relevant measurabilities follow from the explicit form of the $\Ha$ given in Example~\ref{example:two}) $H=L e^{\Phi(q)\cdot}+(H\omega)\star \Wq$. At the same time, for $x\in (-\infty,0]$, $e^{-\Phi(q)x}H(x)=\int \Ha(x)e^{-\Phi(q)x}\nu(d\alpha)\leq 1$ so $H: \mathbb{R}\to (0,\infty)$ has a left tail that is $\Phi(q)$-subexponential. Next, because the support of $\nu$ is bounded from below away from zero, $\omega$ has a subexponential left tail. By Lemma~\ref{lemma} we obtain $H=\frac{L}{\Lo}\Ho$. But we also have $H(1)=1=\Ho(1)$, hence $L=\Lo$, and the proof is complete.
\end{proof}
Another fairly general class of $\omega$ that can be handled with some success is considered in

\begin{remark}
Let $P$ be a real polynomial, $\alpha\in (0,\infty)$ and $c\in \mathbb{R}$. Suppose $\omega(x)=P(x)e^{\alpha x}$ for all $x\in (-\infty,c]$. Then the recursion of Theorem~\ref{theorem}\ref{conv:ii} can, on $(-\infty,c]$, be successively computed in essentially closed form: one obtains algebraic expressions involving only $\psi$ and its higher-order derivatives. This is because, together with the Laplace transform of $\Wq$ \eqref{eq:laplace} that is given in terms of $\psi$, one obtains, by successive differentiation, also expressions for its higher order derivatives. In a similar vein, if $c<0$ and $\omega(x)=P(1/x)e^{\alpha x}$ for $x\in (-\infty,c]$, then one gets iterated integrals involving $\psi$ (cf. Example~\ref{example:new}). 
\end{remark}
Finally, when given a concrete $\Wq$, it may of course very well happen that for a specific form of $\omega$, the convolution equation of Theorem~\ref{theorem}\ref{conv:ii} admits an explicit solution (even as it fails to do so for a general $\Wq$). A flavor of this is given in the next section.

\section{Application to a model for the price of a financial asset}\label{illustration}
Assume $\omega>0$. We consider the process $S$ defined by $$S_t:=\exp(Y_t)\mathbbm{1}_{\{t<\zeta\}},\quad t\in [0,\infty),$$ as a model for the price of a (speculative) financial asset (here $Y$ is the process from \eqref{eq:Y}). When $\omega= 1$, then $\zeta=e_q$, $X=Y$ on $[0,e_q)$, and $S$ is nothing but the classical exponential L\'evy model for the price of a risky asset (defaulted at $\zeta$), see \cite{tankov} for a recent review. The idea with allowing a more general $\omega$ is that the asset price may  ``move faster or slower along its trajectory'', depending on the price level, destroying the stationary independent increments property of the log-returns, but preserving their Markovian character. 

Using time changed L\'evy processes to model financial assets is of course not new, see e.g. \cite{carr,kingler}. We set $\QQ_z:=\PP_{\log z}$, $z\in (0,\infty)$, for convenience. 

Suppose then in this setting that we are interested in the simple problem of the determination of the optimal level $b$ at which to sell the asset, having bought it at the level $z\in (0,\infty)$, under an inflation/impatience rate $\gamma\in [0,\infty)$. In other words, if we let $R_b^+:=\inf\{s\in (0,\infty):S_s>b\}$ denote the first hitting time of the set $(b,\infty)$ by the process $S$, then we would like a solution to the problem 
\begin{equation}\label{eq:apps}
\max_{b\in [z,\infty)}\AA(b)\text{ with } \AA(b):=\QQ_z[S_{R_b^+}\exp\{-\gamma R_b^+ \};R_b^+<\infty]\text{ for } b\in [z,\infty).
\end{equation}
(More generally we may simply be interested in $\AA(b)$ itself, if we are predetermined to sell at the level $b\in [z,\infty)$.) But, for $b\in [z,\infty)$, since $S_{R_b^+}=b$ on $\{R_b^+<\infty\}$,\footnotesize $$\AA(b)=b\QQ_z[\exp\{-\gamma R_b^+ \};R_b^+<\infty]=b\PP_{\log z}[\exp\{-\gamma T_{\log b}^+ \};T_{\log b}^+<\zeta]=b\mathcal{B}^{(\gamma\omega)}_q(\log z,\log b)=b\frac{\Hog(\log z)}{\Hog(\log b)}.$$ \normalsize Hence \eqref{eq:apps} is intimately related to the determination of the quantity \eqref{what-we-are-after}.  

Now  if $\omega= 1$, then, of course, because of the martingale property of $(\exp\{X_t-\psi(1)t\})_{t\in [0,\infty)}$ and the optional sampling theorem, $\AA(b)$ is monotone in $b\in [z,\infty)$, and the problem is trivial. However, for a general $\omega$ this is no longer the case, as we will see on an example shortly.

Let indeed $\alpha\in (0,\infty)$ and $\omega=e^{\alpha \cdot}\land 1$: as a possible rationale for such a choice, one may imagine that the asset moves faster along its trajectory at smaller price levels, reflecting that the investors are then more jittery, increasing the velocity of the trades. In this case  we can be quite explicit about the nature of $\Hog=:H$, as follows. Denote $L:=\Log$, and by $P$ the Patie scale function \eqref{eq:pssMp-exit} of Example~\ref{example:two}, so that $L=(\sum_{k=0}^\infty a_k\gamma^k)^{-1}$ and 

$$
P(x)= Le^{\Phi(q)x}\sum_{k=0}^\infty a_k(\gamma e^{\alpha x})^k,\quad x\in \mathbb{R},$$ where we have set $a_k:=\left(\prod_{l=1}^k(\psi(\Phi(q)+l\alpha)-q)\right)^{-1}$ for $k\in\mathbb{N}_0$. Then $H=P$ on $(-\infty,0]$,  while for $x\in [0,\infty)$,
\begin{equation}\label{eq:apply}
H(x)=h(x)+\gamma\int_0^xH(y)\Wq(x-y)dy,
\end{equation}
where 
\begin{equation}\label{eq:h}
h(x):=Le^{\Phi(q)x}+\gamma\int_{-\infty}^0e^{\alpha y}P(y)\Wq(x-y)dy, \quad x\in [0,\infty).
\end{equation}
As in the proof of Lemma~\ref{lemma} one sees that, on $[0, \infty)$, $H={\uparrow\!\!\!\text{--}\!\lim}_{n\to\infty}H_n$ where $H_0:=h$ and then recursively $H_{n+1}=h+\gamma  H_n\star \Wq$ for $n\in \mathbb{N}_0$, the latter convolution being now on $[0,\infty)$. Taking Laplace transforms on $[0,\infty)$ in \eqref{eq:apply} (denoting them by a hat) and using \eqref{eq:laplace}, it is also true that $\hat{H}(s)=\hat{h}(s)+\gamma \hat{H}(s)/(\psi(s)-q)$ and hence 
\begin{equation}\label{eq:laplace-appl}
\hat{H}=\hat{h}\frac{\psi-q}{\psi-\gamma-q}\text{ on } (\Phi(q+\gamma),\infty)
\end{equation}
(using Theorem~\ref{theorem}\ref{III} and the known solution for $\omega= 1$ it is easily checked that $\hat{H}(s)<\infty$ for $s\in (\Phi(q+\gamma),\infty)$).

Now  if $q+\gamma< \psi(1)$, then a comparison argument (with $\omega= 1$) shows that $\AA(b)\to \infty$ as $b\to\infty$. However, in general, it does not seem obvious how to determine an optimal $b$ analytically, nor  is it our intent to pursue this further here; we content ourselves by demonstrating how $\AA(b)$ may fail to be monotone in $b$. This transpires already in the (presumably simplest) case when $\psi(s)=s^2$, $s\in [0,\infty)$, corresponding to $X$ being a multiple (by the factor $\sqrt{2}$) of Brownian motion, and $q=0$, $\alpha=z=1$. Under the latter specifications, $\Phi(0)=0$ and $\psi(1)=1$, while $W(x)=x$ and (via \eqref{eq:h}) $h(x)=1+b_\gamma x$,  $x\in [0,\infty)$, where $b_\gamma:=\gamma \sum_{k=0}^\infty\frac{\gamma^k}{k!^2(k+1)}/\sum_{k=0}^\infty\frac{\gamma^k}{k!^2}$. Inverting the Laplace transform \eqref{eq:laplace-appl} for $H$ we obtain $$H(x)=\frac{\sqrt{\gamma}-b_\gamma}{2\sqrt{\gamma}}e^{-\sqrt{\gamma}x}+\frac{\sqrt{\gamma}+b_\gamma}{2\sqrt{\gamma}}e^{\sqrt{\gamma}x},\quad x\in [0,\infty) ;$$ see Figure~\ref{fig}. 

\begin{figure}[h!]
\begin{center}
 \includegraphics[scale=.55]{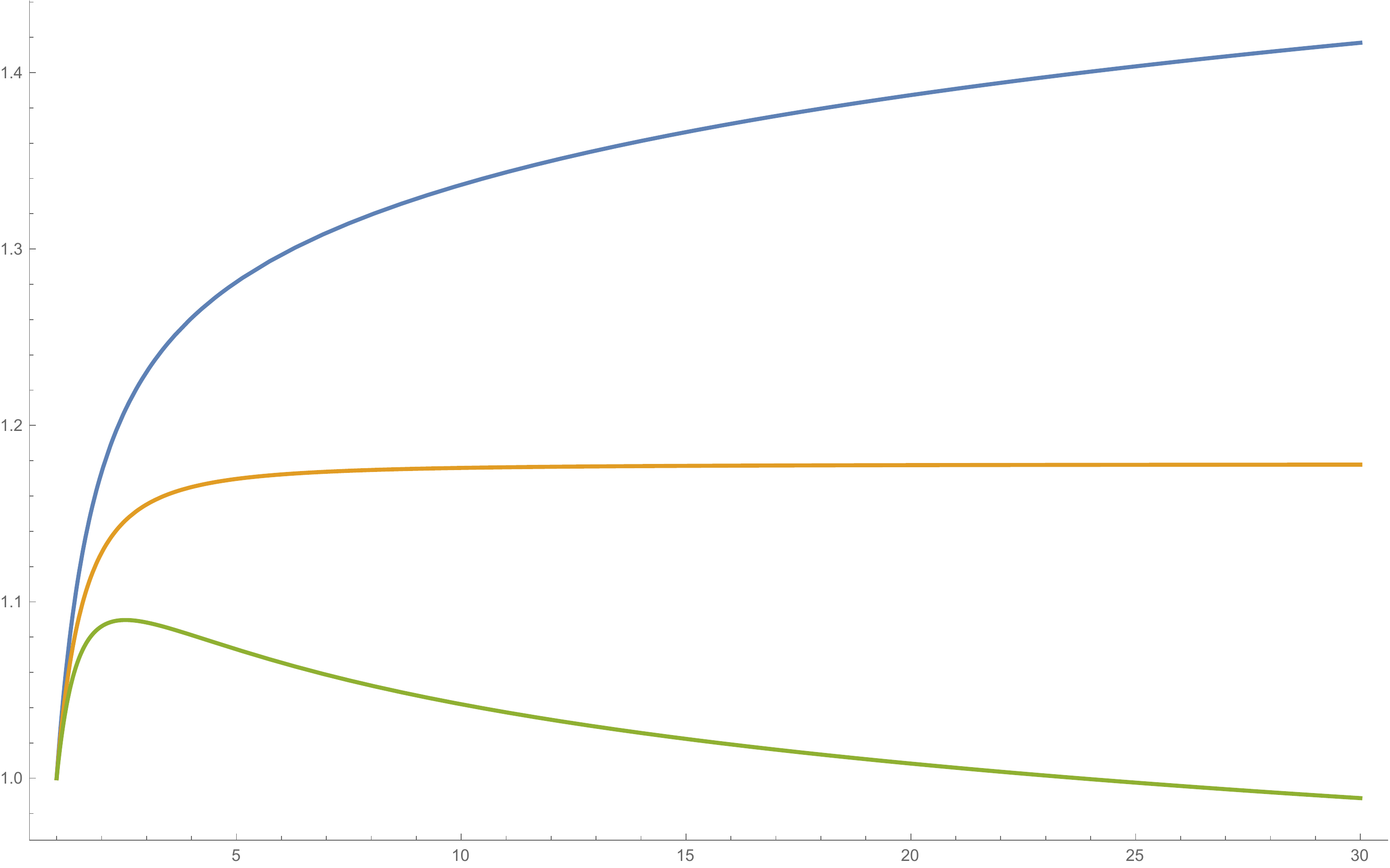}
\caption{The function $\AA$ of \eqref{eq:apps} on the interval $[1,30]$ in the case when $q=0$, $\alpha=z=1$ and $\psi(s)=s^2$, $s\in [0,\infty)$, for three values of the inflation/impatiance parameter $\gamma$, top-to-bottom: $\gamma=0.9$ (blue), $\gamma=1$ (orange) and $\gamma=1.1$ (green). The case $\gamma=1.1$ exhibits non-trivial behavior. Unlike with $\omega= 1$ when the optimal $b$ would be $1$, now the optimal $b$ is strictly greater than $1$. Intuitively this is due to the fact that the clock ``runs faster'' when the price level is small, thus ``buying'' us some time in terms of the inflation depreciation, as we wait for a higher price level.}\label{fig}
\end{center}
\end{figure}

\bibliographystyle{plain}
\bibliography{One-sided-exit-snlp-omega-killing}
\end{document}